\documentclass[a4paper]{article}

\usepackage[latin1]{inputenc} 
\usepackage{setspace} 
\usepackage{amssymb}
\usepackage[]{epsf,epsfig,amsmath,amssymb,amsfonts,latexsym}
\usepackage{amsthm}
\usepackage{mathrsfs}
\usepackage{wasysym}
\usepackage{bbm}		
\usepackage[inline]{enumitem}
\usepackage{nccmath}	
\usepackage{mathtools}
\usepackage{braket}		
\usepackage[normalem]{ulem}		
\usepackage{nicefrac}
\usepackage{cancel}
\usepackage{xcolor}
\usepackage[a4paper, left=2.7cm, right=2.7cm, top=3.5cm, bottom=2cm]{geometry}
\usepackage[pdftex=true,hyperindex=true,pdfborder={0 0 0}]{hyperref}
\usepackage[nottoc]{tocbibind}	
\usepackage{doi}
\usepackage[numbers]{natbib}
\usepackage{cleveref}
\usepackage{tikz}
\usetikzlibrary{shapes,arrows,automata,matrix,fit,patterns}
\usepackage{soul} 

\usepackage{autonum}		

\theoremstyle{plain}
\newtheorem{theorem}{Theorem}

\newtheorem{corollary}[theorem]{Corollary}
\newtheorem{proposition}[theorem]{Proposition}

\theoremstyle{definition}

\newtheorem{example}[theorem]{Example}

\newcommand{\exampleqed}{\ensuremath{\ocircle}\par}

\newcommand{\ZZ}{\mathbb{Z}}			
\newcommand{\QQ}{\mathbb{Q}}			

\newcommand{\symb}[1]{\mathtt{#1}}		
\makeatletter
\newcommand{\vast}{\bBigg@{4}}
\newcommand{\Vast}{\bBigg@{5}}

\makeatother
	
\newcommand{\define}[1]{\textbf{#1}}

\newcommand{\xConfig}[1]{%
	\begin{tikzpicture}[
		baseline=-\the\dimexpr\fontdimen22\textfont2\relax,ampersand replacement=\&]
		\matrix[
			matrix of math nodes,
			nodes={
				minimum size=1.4ex,text width=1.4ex,
				text height=1.4ex,inner sep=3pt,draw={gray!20},anchor=center
			}, row sep=1pt,column sep=1pt
		] (config) {#1};
		\node[draw,rectangle,help lines,gray!50, dashed,fit=(config),inner sep=-1pt] {};
	\end{tikzpicture}
}

\title{%
	Aperiodic subshifts of finite type on groups which are not finitely generated
}

\author{Sebasti\'an Barbieri}

\newcommand{\Addresses}{{
		\bigskip

		\hskip-\parindent   S.~Barbieri, \textsc{Departamento de Matem\'{a}tica y ciencia de la computaci\'{o}n, Universidad de Santiago de Chile, Santiago. Chile.}\par\nopagebreak
		\textit{E-mail address}: \texttt{sebastian.barbieri@usach.cl}
}}

\date{}

\begin{document}

\maketitle

\begin{abstract}
	We provide an example of a non-finitely generated group which admits a nonempty strongly aperiodic SFT. Furthermore, we completely characterize the groups with this property in terms of their finitely generated subgroups and the roots of their conjugacy classes. 
	
	\medskip
	
	
	\smallskip
	
\end{abstract}


Let $A$ be a finite set and $G$ a group with identity $1$. Consider the set of configurations $A^{G} = \{x \colon G \to A\}$ endowed with the left shift action $G \curvearrowright A^{G}$ given by \[ (gx)(h) = x(g^{-1}h) \mbox{ for every } x \in A^{G} \mbox{ and } g,h \in G.  \]

A set $X \subset A^{G}$ is called a \define{subshift of finite type} (SFT) if there exists a finite set $F \subset G$ and $L \subset A^{F}$ such that $x \in X$ if and only if for every $g \in G$ we have that the restriction of $gx$ to $F$ is in $L$. An SFT can be though of as a set of colorings of $G$ using colors from $A$ which satisfy a finite set of local rules encoded by $L$. 


An SFT $X$ is called \define{strongly aperiodic} (SA) if the shift action is free, namely, given $x \in X$ we have that $gx = x$ can only hold for $g = 1$. Clearly $X=\varnothing$ is SA, but this case is not very interesting. This definition begs the question of whether nonempty SA SFTs actually exist and the answer turns out to depend on which group $G$ we are considering. 

Let us illustrate the failure of the existence of nonempty SA SFT in the case $G = \ZZ$. Given $A,F$ and $L$ which define an SFT $X$ we may assume without loss of generality that $F$ is of the form $\{0,\dots,n-1\}$ for some $n \geq 1$. Consider the finite directed graph with vertices $A^{n+1}$ and such that there is an edge from $u_0\dots u_n$ to $v_0\dots v_n$ if and only if $u_1\dots u_n = v_0\dots v_{n-1} \in L$. It is clear that the elements of $X$ are precisely the projections of the bi-infinite walks in this graph to the first symbol in each coordinate (for further details, see~\cite{LindMarcus_2021_book}). If $X$ is nonempty, it follows that the graph must contain a cycle, say of length $m>0$, and consequently there must exist $x \in X$ for which $mx = x$, showing that $X$ is not SA.

The case of $G = \ZZ^2$ is much more interesting. In~\cite{Wang1961} Wang studied the algorithmic problem of deciding, given as inputs $A,F$ and $L$ whether the SFT $X$ they define is empty or not\footnote{To be rigorous, Wang studied an equivalent problem where the SFT is described by square tiles with colored edges and in any tiling of $\ZZ^2$ the colors at the edges must match.}. Wang did not solve this problem but showed that if $\ZZ^2$ does not admit any nonempty SA SFT, then the problem would be algorithmically decidable. Years later Berger solved the problem~\cite{Berger1966} showing that it was algorithmically undecidable, and in doing so constructed the first known example of a nonempty SA SFT in $\ZZ^2$. After this breakthrough, several beautiful SA SFTs on $\ZZ^2$ have been constructed, see~\cite{JeandelRao11Wang,Kari1996259,Robinson1971}.

In recent years the problem of classifying the groups which admit a nonempty SA SFT has gained notoriety. This has led to wonderful discoveries that relate this property with the algorithmic and geometric properties of groups. For instance, the property of admitting a nonempty SA SFT is invariant under commensurability~\cite{CarrollPenland} and under quasi-isometries of finitely presented groups~\cite{Cohen2014}. It is also known that finitely generated groups with infinitely many ends (in particular, virtually free groups) cannot admit nonempty SA SFTs~\cite{Cohen2014}, and that finitely generated and recursively presented groups which admit SA SFTs must necessarily have decidable word problem~\cite{Jeandel2015}. Within the groups that satisfy these constraints, several are known to admit SA SFTs. For instance polycyclic-by-finite groups which are not virtually cyclic~\cite{BallierStein2018,Jeandel2015_2}, one-ended word-hyperbolic groups~\cite{CohenGoodmanS2015,CohenGSRieck_2021}, some Baumslag-Solitar groups~\cite{AubrunKari2013,AubrunKari2021,EsnayMoutot2022}, groups of the form $\ZZ^d \rtimes_{\varphi} G$ with $d\geq 2$, $\varphi \in \operatorname{GL}_d(\ZZ)$ and $G$ infinite, finitely generated and with decidable word problem~\cite{BS2018}, groups which are the direct product of three infinite, finitely generated groups with decidable word problem~\cite{Barbieri_2019_DA}, the Grigorchuk group, and more generally, any finitely generated branch group with decidable word problem~\cite{Barbieri_2019_DA}, and any self-simulable group with decidable word problem such as Thompson's $V$, $\operatorname{GL}_n(\ZZ)$ and $\operatorname{SL}_n(\ZZ)$ for $n \geq 5$, or the direct product of any two finitely generated non-amenable groups with decidable word problem~\cite{BaSaSa_2021}.

A common point shared by all the existence results mentioned above is that they apply to finitely generated groups. For obvious reasons this is not very surprising: if $X\subset A^G$ is an SFT on a group $G$ described by $L\subset A^F$ and $H = \langle F \rangle$ is the subgroup of $G$ finitely generated by $F$, we may consider instead $Y \subset A^H$ as the SFT which is defined by $L$ but on $H$. It turns out that configurations in $X$ can be seen as independent copies of configurations of $Y$ on each coset of $H$. 

\begin{proposition}\label{prop}
	Let $H \leqslant G$, $F\subset H$ finite, $L\subset A^F$ and let $X,Y$ be the SFTs defined by $L$ on $G$ and $H$ respectively. Let $(g_i)_{i \in I}$ be a set of left coset representatives of $H$ in $G$. Then $x \in X$ if and only if for every $i \in I$ we have $y_i = (x(g_ih))_{h \in H} \in Y$.
\end{proposition}

\begin{proof}
	Let $x \in X$ and $i \in I$. For every $u \in H$ we have $(uy_i)(h) = y_i(u^{-1}h) = x(g_iu^{-1}h) = (ug_i^{-1}x)(h)$. As $x \in X$, it follows that $(ug_i^{-1}x)|_F \in L$ and thus $(uy_i)|_F \in L$ and $y_i \in Y$. Conversely, suppose $y_i \in Y$ for every $i \in I$ and let $g \in G$. We may choose $i \in I$ and $v \in H$ such that $g^{-1} = g_iv$ and thus we have $(gx)(h) = (v^{-1}g_i^{-1}x)(h) = x(g_ivh) = y_i(vh) = (v^{-1}y_i)(h)$ for every $h \in H$. As $y_i \in Y$, it follows that $v^{-1}y_i|_F \in L$ and thus $gx|_F \in L$, showing that $x \in X$.
\end{proof}
%


Based on the former argument, one should expect that non-finitely generated groups would have a hard time admitting nonempty SA SFTs, as one could always pick the same configuration on each coset and use that to create periodicity. Let us exemplify that this intuition is wrong with a real world situation.

Let us suppose the reader is walking on the street, while suddenly they get jumped by a thief who demands: ``Quickly! give me an example of a non-finitely generated group which admits a nonempty SA SFT or I shall take your wallet!''. The reader, scared by the gravity of the situation, might wrongly answer that such examples do not exist by explaining the argument above. The thief, with a victorious grin on their face, would take out their portable blackboard and write down the following example.

\begin{example}
	Consider $G = \QQ^2$ and let $Y\subset A^{\ZZ^2}$ be a nonempty SA SFT. Let $X \subset A^{\QQ^2}$ be given by the condition \[ x \in X \mbox{ if and only if } (sx)|_{\ZZ^2} \in Y, \mbox{ for every } s \in \QQ^2.  \]
	As $Y$ is a nonempty SFT, it follows that $X$ is a nonempty SFT. For $q \in \QQ^2$, we may write $q = \left( \frac{p_1}{r_1},\frac{p_2}{r_2} \right)$ with $p_1,p_2,r_1,r_2 \in \ZZ$ and $r_1r_2\neq 0$. It follows that $r_1r_2q \in \ZZ^2$. Now suppose we have $qx = x$, then we also have $(r_1r_2q)x = x$ and from the fact that $x|_{\ZZ^2} \in Y$, we obtain that necessarily $r_1r_2q = (0,0)$. As $r_1r_2 \neq 0$, we deduce that $q = (0,0)$ and thus $X$ is SA.\hfill\exampleqed
\end{example}


Faced with this example, the reader would have no other choice but to surrender their wallet to the thief. The argument for nonexistence sketched above is incomplete: while it is true that one may choose independently any configuration in every coset, it may happen that powers of the coset representatives end up inevitably falling on a non-trivial element of the finitely generated subgroup and thus destroy any global translational symmetry. As we shall show, it turns out that modulo conjugacy this is essentialy the sole way that SA SFTs can arise in non-finitely generated groups. In fact, we shall show that global aperiodicity may arise even if the subshift in the finitely generated subgroup is not strongly aperiodic.

For an SFT $X\subset A^{G}$, define its \define{free part} as the set \[ \operatorname{Free}(X) = G \setminus \bigcup_{x \in X}\operatorname{Stab}_G(x) = \{ g \in G : gx \neq x \mbox{ for every } x \in X\}.  \]

In particular, a nonempty SFT $X$ is SA if and only if $\operatorname{Free(X)}= G \setminus \{1\}$. For a subset $M\subset G$, let us denote the set of its \define{roots} in $G$ by \[ R_G(M) = \{ g \in G : \mbox{ there is } n >0 \mbox{ such that } g^n \in M \}. \]
Finally, given $g \in G$ let us denote its conjugacy class by $\operatorname{Cl}(g) = \{ tgt^{-1} : t \in G \}$.



\begin{proposition}\label{prop2}
	Let $H \leqslant G$, $F\subset H$ finite, $L\subset A^F$ and let $X,Y$ be nonempty SFTs defined by $L$ on $G$ and $H$ respectively. We have that $g \in \operatorname{Free}(X)$ if and only if 
	\[ \operatorname{Cl}(g) \cap R_G(\operatorname{Free}(Y)) \neq \varnothing.  \]
\end{proposition}

\begin{proof}

	If $g \notin \operatorname{Free}(X)$, then there exists $x \in X$ such that $gx=x$. Suppose there is $n >0$, $h \in H$ and $t \in G$ such that $g^n =t^{-1}ht$. It follows that we would have $t^{-1}htx = g^nx = x$ and thus $htx = tx$. Letting $z = tx$, it follows that $hz = z$. As $z \in X$, it follows that if we let $y=z|_{H} \in Y$ we have $hy = y$ and thus $h \notin \operatorname{Free}(Y)$. This shows that $\operatorname{Cl}(g) \cap R_G(\operatorname{Free}(Y)) = \varnothing$.
	
	
	Conversely, let $g \in G$ be such that $\operatorname{Cl}(g) \cap R_G(\operatorname{Free}(Y)) = \varnothing$ and choose (using the axiom of choice) a set $(\ell_i)_{i \in I}$ of left coset representatives of $H$. There is a unique well-defined permutation $\varphi \colon I \to I$ which satisfies $g\ell_i = \ell_{\varphi(i)}h$ for some $h \in H$. In particular, $\varphi$ induces an action of $\ZZ$ on $I$ given by $m \cdot i = \varphi^m(i)$ for $m \in \ZZ$ and $i \in I$.
	
	Choose (using the axiom of choice) $J \subset I$ such that it contains exactly one representative of each orbit of $\varphi$, that is, for every $i \in I$ there is a unique $j \in J$ such that $i = \varphi^m(j)$ for some (non necessarily unique) $m \in \ZZ$. For $n \in \ZZ$ and $j \in J$ let $h_{n,j} = \ell_{\varphi^n(j)}^{-1}g^{n}\ell_j \in H.$

	We will now define a configuration $x \in A^G$. Let $y^* \in Y$ be a fixed configuration and let $j \in J$, there are two cases to consider:
	
	\begin{enumerate}
		\item If $\{ \varphi^m(j) \}_{m \in \ZZ}$ is infinite, we let $y_j = y^*$ and define $x(\ell_{\varphi^m(j)} s) = h_{m,j}y_j(s) = y_j( h^{-1}_{m,j}s)$ for every $s \in H$ and $m \in \ZZ$.
		\item If $\{ \varphi^m(j) \}_{m \in \ZZ}$ is finite, let $n >0$ be the least positive integer such that $\varphi^n(j)=j$. As $h_{n,j} = \ell_{j}^{-1}g^{n}\ell_j \in \operatorname{Cl}(g^n)$ it follows by our assumption that $h_{n,j}\notin \operatorname{Free}(Y)$. Therefore there exists $y_j \in Y$ such that $h_{n,j}y_j = y_j$. We define $x(\ell_{\varphi^m(j)} s) = h_{m,j}y_j(s) = y_j( h_{m,j}^{-1}s)$ for every $s \in H$ and $m \in \ZZ$. Notice that as $h_{n,j}y_j = y_j$, this is well-defined.
	\end{enumerate}
	
	By construction, we have that $(x(\ell_i h))_{h \in H} \in Y$ for every $i \in I$ and thus by~\Cref{prop} we have that $x \in X$. Let us show that $gx = x$. Indeed, for $j \in J, m \in \ZZ$ and $s \in H$ we have
	\[ gx(\ell_{\varphi^n(j)}s)  = x(g^{-1}\ell_{\varphi^n(j)}s) = x(\ell_{\varphi^{n-1}(j)}\ell_{\varphi^{n-1}(j)}^{-1}g^{-1}\ell_{\varphi^n(j)}s) \]
	As $\ell_{\varphi^{n-1}(j)}^{-1}g^{-1}\ell_{\varphi^n(j)} \in H$, we obtain that $x(\ell_{\varphi^{n-1}(j)}\ell_{\varphi^{n-1}(j)}^{-1}g^{-1}\ell_{\varphi^n(j)}s) = y_j(h_{n-1,j}^{-1}\ell_{\varphi^{n-1}(j)}^{-1}g^{-1}\ell_{\varphi^n(j)}s)$. From here we obtain that \[ gx(\ell_{\varphi^n(j)}s)  = y_j(h_{n-1,j}^{-1}\ell_{\varphi^{n-1}(j)}^{-1}g^{-1}\ell_{\varphi^n(j)}s)  = y_j(h_{n,j}^{-1}s)  = x(\ell_{\varphi^n(j)}s)\]
	
	This shows that $gx = x$ and thus we conclude that $g \notin \operatorname{Free}(X)$. \end{proof}

The previous proposition tells us that we may determine whether a group admits a nonempty SA SFT by just looking at its finitely generated subgroups.

\begin{theorem}\label{thm}
	A group $G$ admits a nonempty strongly aperiodic subshift of finite type if and only if there exists a finitely generated subgroup $H \leqslant G$ and a nonempty SFT $Y\subset A^{H}$ such that for every $g \in G \setminus \{1\}$ we have 
	\[ \operatorname{Cl}(g) \cap R_G(\operatorname{Free}(Y)) \neq \varnothing.  \]
\end{theorem}

\begin{proof}
	If $X\subset A^G$ is a nonempty SA SFT given by $F\subset G$ finite and $L\subset A^F$, then $H =\langle F \rangle$ is a finitely generated subgroup and by~\Cref{prop2} for every $g \in \operatorname{Free}(X) = G \setminus \{1\}$ we have $\operatorname{Cl}(g) \cap R_G(\operatorname{Free}(Y)) \neq \varnothing$, where $Y\subset A^H$ is the nonempty SFT given by $L$ on $H$. Conversely, if $H\leq G$ is a finitely generated subgroup, $Y\subset A^H$ a nonempty SFT given by $F\subset H$ finite and $L\subset A^F$, and such that for every $g \in G \setminus \{1\}$ we have $\operatorname{Cl}(g) \cap R_G(\operatorname{Free}(Y)) \neq \varnothing$, then again by~\Cref{prop2} we have that $G\setminus \{1\}\subset \operatorname{Free}(X)$ for the nonempty SFT $X\subset A^G$ induced by $L$ on $G$. Thus $X$ is SA.
\end{proof}

\begin{corollary}
	Let $G$ be a group and $H \leqslant G$ a finitely generated subgroup which admits a nonempty strongly aperiodic SFT and such that for every $g \in G\setminus \{1\}$ we have $\operatorname{Cl}(g) \cap R_G(H \setminus \{1\}) \neq \varnothing$. Then $G$ admits a nonempty strongly aperiodic SFT.
\end{corollary}

Notice that for any nonempty subshift $Y\subset A^H$ we have $\operatorname{Free}(Y)\subset H\setminus \{1\}$ and thus the condition $\operatorname{Cl}(g)\cap R_G(\operatorname{Free}(Y)) \neq \varnothing$ always implies that $\operatorname{Cl}(g)\cap R_G(H \setminus \{1\}) \neq \varnothing$. This trivial remark suggests the question of whether there are any examples where $X$ is SA but $Y$ is not. The following example due to Salo (also in an article by Jeandel~\cite{Jeandel2015} with a slightly different construction) shows that this case can indeed occur.

\begin{example}
	A result of Osin~\cite{Osin2010} shows that every countable torsion-free group can be embedded in a $2$-generated group with exactly two conjugacy classes. In particular there is a $2$-generated group $G$ with two conjugacy classes and such that there is $t \in G$ with $\langle t \rangle \simeq \ZZ$. Let $A = \{\symb{a},\symb{b}\}$, $F = \{1,t\}$ and $L = \{ (1\mapsto \symb{a}, t \mapsto \symb{b}),(1\mapsto \symb{b}, t \mapsto \symb{a})\}$, that is, we let $X\subset\{\symb{a},\symb{b}\}^G$ be the SFT which consists of all maps for which $x(gt) \neq x(g)$ for every $g \in G$.
	
	Let $Y \subset \{\symb{a},\symb{b}\}^{\langle t \rangle}$ be the SFT defined by $F$ and $L$ on $\langle F \rangle \simeq \ZZ$. On the one hand, we have that $Y$ is nonempty and consists on the two periodic configurations which alternate symbols, thus $\operatorname{Free}(Y) = \{ t^{2n+1} : n \in \ZZ\}$ consists of the odd powers of $t$, in particular $Y$ is not SA. On the other hand, as there is only one non-trivial conjugacy class, we have that for every $g \in G \setminus \{1\}$, $t \in \operatorname{Cl}(g)\cap R_G(\operatorname{Free}(Y))$ which is nonempty, and thus $X$ is SA by~\Cref{thm}.\hfill\exampleqed
\end{example}

\textbf{Acknowledgments}: The author is grateful to Ville Salo for finding a subtle error in the first version of this paper and suggesting the Osin group example. The author also wishes to thank the referees for suggesting various improvements. The author was supported by the FONDECYT grant 11200037.
\Addresses

\bibliographystyle{abbrv}
\bibliography{ref}

\end{document}